\documentclass[a4paper,12pt]{article}

\usepackage[T1]{fontenc}
\usepackage{amsmath,amssymb,amsfonts,amsthm,mathrsfs,eucal,dsfont,verbatim}

\evensidemargin
-1cm

\usepackage{color}%for comments

\theoremstyle{plain}
\newtheorem{theorem}{Theorem}
\newtheorem{lemma}{Lemma}
\newtheorem{proposition}{Proposition}

\theoremstyle{definition}
\newtheorem{remark}{Remark}
\newtheorem{definition}{Definition}
\newtheorem{example}{Example}

\newcommand{\real}{\mathds{R}}
\newcommand{\rn}{{\mathds{R}^n}}
\newcommand{\Cc}{\mathds{C}}

\newcommand{\I}{\mathds{1}}

\newcommand{\Sss}{\mathbb{S}}
\newcommand{\LL}{\mathcal{L}}

\newcommand{\RR}{\mathrm{Re}\,}

\numberwithin{equation}{section}

\begin{document}

%\selectlanguage{english}

\title{Compound kernel estimates for the transition probability density of a L\'evy process in $\rn$}

\author{%
    \textsc{Victoria Knopova}%
    \thanks{  V.M.\ Glushkov Institute of Cybernetics,
            NAS of Ukraine,
            40, Acad.\ Glushkov Ave.,
            03187, Kiev, Ukraine,
            \texttt{vic\underline{ }knopova@gmx.de}}}

\date{ }
\maketitle

\begin{abstract}

    We construct in the small-time setting  the upper and lower estimates for the transition probability density  of  a L\'evy process in $\rn$. Our approach relies on the complex analysis technique and the asymptotic analysis of the inverse Fourier transform of the characteristic function  of the respective process.

    
\end{abstract}

\maketitle
\numberwithin{equation}{section}

\section{Introduction}

Let   $Z_t$ be  a real-valued L\'evy process in $\rn$ with characteristic exponent $\psi$, i.e.
$$
Ee^{i \xi \cdot Z_t} = e^{-t\psi(\xi)},\quad \xi\in \rn.
$$
It is known that the characteristic exponent  $\psi$ admits the L\'evy-Khinchin representation
\begin{equation}
    \psi(\xi)=ia\cdot \xi - \frac{1}{2}  \xi \cdot  Q \xi +\int_\rn (1-e^{i\xi\cdot  u}+i\xi\cdot  u\I_{\|u\|<1})\mu(du), \label{psi}
    \end{equation}
where $a\in \rn$, $Q$ is a positive  semi-definite $n \times n$ matrix,  and $\mu$ is a L\'evy measure, i.e. $\int_\rn(1\wedge \|u\|^2)\mu(du)<\infty$.  In what follows we assume that $Q\equiv 0$, and
\begin{equation}\label{mu}
\mu(\rn)=\infty.
\end{equation}
Clearly, \eqref{mu}   is necessary for $Z_t$ to possess a distribution density.

 In the past decades such questions as the existence and  properties of the transition probability density of L\'evy and, more generally, Markov processes,  attracted a lot of attention. Although some progress is already achieved, this problem is highly non-trivial.   One  can  prove the existence of the transition probability density of a symmetric Markov process and study its properties by applying  the Dirichlet form technique, see  \cite{BBCK09},
 \cite{CKK11}, \cite{C09},
 \cite{BGK09},
\cite{CK03},  \cite{CK08},  \cite{CKK09}. The other approach relies  on versions of the   Malliavin calculus for jump processes, see  \cite{Le87}, \cite{Ish94}--\cite{Ish01},  \cite{Pi96}--\cite{Pi97b}, and provides the pointwise small-time asymptotic of the transition probability density of a Markov process which is a solution to a L\'evy-driven SDE.   Under certain assumptions on the L\'evy measure estimates on the transition probability density are obtained in \cite{KS13a}-\cite{KS13b}, see also  the references therein  for earlier results.  In \cite{KK12a}, which is the one-dimensional predecessor  of the current paper, we investigated the transition probability density  $p_t(x)$ of a L\'evy process, and proposed  a specific form of estimates, which we call the \emph{compound kernel estimates}, see Definition~\ref{def1} below. The approach described in \cite{KK12a} relies on the asymptotic analysis of the inverse Fourier transform of the respective characteristic function.  The analysis made in \cite{KK12a}  shows that under rather general  assumptions   the \emph{bell-like estimate}
\begin{equation}
p_t(x)\leq \sigma_t g(\|x\|\sigma_t)\label{bell}
\end{equation}
 where $ g\in L_1(\rn)$, and $\sigma_t$ is some "scaling function",  is not possible. We also point out, that in the case of a L\'evy process the results obtained in  \cite{Pi96}--\cite{Pi97b} and \cite{Ish01} fit in our  observation.  At the same time, the  upper  and lower compound kernel estimates   give an adequate picture of behaviour of the transition probability density. In \cite{KK13a}, \cite{KK13b}  we investigate possible applications  of the compound kernel estimates for the construction of the transition probability density of some class of Markov processes.

In this paper we investigate the transition probability density of a L\'evy process in the  multi-dimensional  setting.
 In Section~\ref{set} we set the notation and formulate  our main result  Theorem~\ref{t-main1}. Section~\ref{proofs} is devoted to the proof of Theorem~\ref{t-main1}. In Section~\ref{subb}, Theorems~\ref{Omey} and \ref{den}, we treat the particular cases in which  it is possible to construct a bell-like estimate \eqref{bell}. In Section~\ref{exam} we illustrate our results by examples. As  already  mentioned, even if one can construct an estimate of the form \eqref{bell}, it may prove to be not informative.   In particular, in Example~\ref{exa2} we consider the discretized analogue of an $\alpha$-stable L\'evy measure, and show that in the multi-dimensional setting the bell-like estimate for the respective transition probability density, which is  given by Theorem~\ref{Omey}, is not integrable in $x$. At the same time, the compound kernel estimate provided by Theorem~\ref{t-main1} gives an adequate answer.

\section{Settings and the main result}\label{set}

\textbf{Notation:} We denote by $\Sss^n$ a unit sphere in $\rn$;
$\xi\cdot \eta$ and  $\|\xi\|$ denote, respectively, the scalar
product of $\xi,\eta\in \rn$ and the Euclidean norm  of $\xi$ in
$\rn$. We write $f\asymp g$ if there exist constants $c_1,c_2>0$
such that $c_1 f(x)\leq g(x)\leq c_2f(x)$ for all $x\in \real$;
$a\wedge b :=\min(a,b)$.

\medskip

To formulate the regularity assumption on the characteristic exponent $\psi$ we  introduce some auxiliary functions.
For $x\in \real$ put
 \begin{equation}\label{UL}
L(x):=x^2\I_{\{|x|<1\}}, \quad U(x):=x^2\wedge 1,
\end{equation}
and define for $\xi\in \rn$ the functions
\begin{equation}\label{psiUL}
\begin{split}
\psi^L(\xi)&:= \int_\rn L(\xi\cdot u)\mu(du)=\int_{ |(\xi\cdot u) |\leq 1} (\xi \cdot u)^2 \mu(du),
\\
\psi^U(\xi):&=\int_{\rn}U(\xi \cdot u)\mu(du)=\int_{\rn} \left((\xi\cdot  u)^2\wedge 1\right)\mu(du).
\end{split}
\end{equation}
Observe that we always have
\begin{equation}
(1-\cos1) \psi^L(\xi)\leq \RR \psi(\xi)\leq 2 \psi^U(\xi).\label{eq1}
\end{equation}

In addition, we assume that  functions $\psi^L$ and $\psi^U$ are comparable, i.e. the assumption below holds  true.

\medskip

\textbf{A.} There exists $\beta>1$ such that $\sup_{l\in \Sss^n}\psi^U(r l)\leq \beta \inf_{l\in \Sss^n} \psi^L(r l)$ for all $r$  large enough.

\medskip

In particular, assumption  \textbf{A} implies the existence of the transition probability density for $Z_t$, see Lemma~\ref{growth}  in Section~\ref{proofs}.

Define
$$
\psi^* (r) := \sup_{l\in \Sss^n}\psi^U(r l),
$$
 and
\begin{equation}
\rho_t:= \inf\{ r:\quad \psi^*(r) =1/t\}.\label{rho1}
\end{equation}
%(Since $\psi^*(r)$ is continuous, one can write $\min$ instead of $\inf$).

We decompose  $Z_t$ into a sum
\begin{equation}\label{decomp}
Z_t=\bar Z_t+\hat Z_t-a_t,
\end{equation}
where

$\bullet$    $a_t\in \rn$ is a vector with coordinates
\begin{equation}
(a_t)_i =t\left(a_i+\int_{1/\rho_t<\|u\|<1}u_i\, \mu(du)\right),\label{ait}
\end{equation}
where the vector  $a\in \rn$  is that from representation  \eqref{psi}, and $\rho_t$ is defined in \eqref{rho1};

$\bullet$ for each $t>0$ the random variables  $\bar Z_t$ and $\hat
Z_t$ are independent; the  variable  $\bar Z_t$ is infinitely
divisible for each $t>0$, with respective characteristic exponent
\begin{equation}
    \psi_t(\xi):=t\int_{\rho_t \|u\|\leq 1} (1-e^{i\xi \cdot u}+i\xi \cdot u)\mu(du), \label{psit}
    \end{equation}
and  $\hat Z_t$ admits   for each $t>0$ the compound Poisson
distribution with the intensity measure
\begin{equation}
\Lambda_t(du):=t  \mu(du)1_{\{ \rho_t \|u\|>1\}}. \label{lam}
\end{equation}
 If  condition \textbf{A} is satisfied, then $\bar Z_t$ possesses a distribution density (see Lemma~\ref{le2} below), which we denote by $\bar p_t(x)$. Therefore, we can represent  $p_t(x)$ as
\begin{equation}
p_t(x)=(\bar p_t* P_t*\delta_{-a_t})(x),\label{conv1}
\end{equation}
where
\begin{equation}
    P_t(dy):=e^{-\Lambda_t(\rn)} \sum_{m=0}^\infty \frac{1}{m!} \Lambda_t^{*m}(dy),\label{Poist}
    \end{equation}
 and  $\Lambda_t^{*m}$ denotes the $m$-fold convolution of the measure $\Lambda_t$; by  $\Lambda_t^{*0}$ we understand  the $\delta$-measure at $0$.

We are looking for a specific  form of the estimate for $p_t(x)$, called the \emph{compound kernel estimate}, see the definition below.

\begin{definition}\label{def1}
Let $\sigma, \zeta: (0, \infty)\to \real$, $h:\rn\to \real$ be some functions, and $(Q_t)_{t\geq 0}$ be a family of finite measures on the Borel $\sigma$-algebra in $\rn$.  We say that a real-valued function $g$ defined on a set $A\subset (0, \infty)\times \rn$ satisfies the \emph{upper compound kernel estimate} with parameters $(\sigma_t, h, \zeta_t, Q_t)$, if
\begin{equation}\label{upperCCE}
g_t(x)\leq\sum_{m=0}{1\over m!}\int_{\rn}\sigma_t h((x-y)\zeta_t)Q_t^{*m}(dy), \quad (t,x)\in A.
\end{equation}
 If the analogue of \eqref{upperCCE} holds true with the sign $\geq$ instead of $\leq$, then we say that the function $g$ satisfies the \emph{lower  compound kernel estimate} with parameters $(\sigma_t, h, \zeta_t, Q_t)$.
\end{definition}
Let us put a  lexicographical order on $\rn$; namely, we say that $x\leq y$, $x=(x_1,\ldots, x_n)$, $y=(y_1,\ldots, y_n)\in \rn$, if there exists $1\leq m\leq n$, such that for all $i<m$ either $x_i=y_i$, or $x_i<y_i$.
Introducing such an order, we can define in the lexicographical sense the first  argument of maximum $x_t$ of the function   $\bar p_t(x)$. Below we show that
 $x_t$ indeed exists, and for every $t_0>0$ there exists $L=L(t_0)$ such that
$$
\|x_t\|\leq L/\rho_t\quad  t\in (0, t_0].
$$
Below we present   our main result on the behaviour of the transition probability density of a L\'evy process in $\rn$.
 \begin{theorem}\label{t-main1} Suppose that condition \textbf{A}   is satisfied.  Then for every $t_0>0$ there exist constants $b_i>0$, $i=1\ldots 4$,  such that the statements below hold true.

 \begin{itemize}\item[I.]    The function
$$
 p_t(x+a_t), \quad (t, x)\in (0, t_0]\times \rn,
$$
satisfies the upper compound kernel estimate with parameters $(\rho_t^n, f_{upper}, \rho_t, \Lambda_t)$, where
\begin{equation}
f_{upper}(x)=b_1 e^{-b_2\|x\|}.\label{fup}
\end{equation}

\item[II.]   The function
$$
 p_t(x+a_t-x_t), \quad (t, x)\in (0, t_0]\times \rn,
$$
satisfies the lower compound kernel estimate with parameters $(\rho_t^n, f_{lower}, \rho_t, \Lambda_t)$, where
\begin{equation}
f_{lower}(x)=b_3 \I_{\|x\|\leq b_4}.\label{flower}
\end{equation}

 \end{itemize}
\end{theorem}
One can obtain in the same fashion as in the statement I of the preceding theorem that $p_t(\cdot)\in C^\infty_b(\rn)$, and construct the upper estimates for  derivatives.
\begin{proposition}\label{prop1}
Suppose that condition \textbf{A}   is satisfied.  Then  there exist
constants $b_1, b_2>0$ such that  for any $N\geq 1$,  $k_i\geq 0$,
$i=1\ldots n$, such that $k_1+\dots+k_n=N$, the function
$$
\left|\frac{\partial^N}{\partial x_1^{k_1} \dots \partial x_n^{k_n}}
p_t(x+a_t) \right|, \quad (t, x)\in (0, t_0]\times \rn,
$$
satisfies the upper compound kernel estimate with parameters $(\rho_t^{n+N}, f_{upper}, \rho_t, \Lambda_t)$.
\end{proposition}
Clearly, in the  case of a symmetric L\'evy measure and a zero drift  the statement of Theorem~\ref{t-main1} holds true with  $a_t = x_t=0$. Moreover, one can get the sharper upper  estimate for $p_t(x)$ and its derivatives.

 \begin{proposition}\label{aux-theo-sym} Suppose that  the process $Z_t$ is symmetric,  and condition \textbf{A}   holds true.
Then the  first  statement of Theorem~\ref{t-main1}  and Proposition~\ref{prop1} hold true  with $a_t$ replaced by zero, and $f_{upper}$ replaced  by
\begin{equation}
f_{upper}(x)=b_1 e^{-b_2 \|x\|\ln (\|x\|+1)}. \label{fup-sym}
\end{equation}
\end{proposition}

\section{Proofs}\label{proofs}

We start with the proof of the auxiliary lemma on the growth of $\psi^U$.

\begin{lemma}\label{growth}
Under condition \textbf{A} we have for $\|\xi\|$ large enough
\begin{equation}
\psi^U(\xi) \geq c \|\xi\|^{2/\beta}, \label{eq-gr}
\end{equation}
where $c>0$ is some constant.
\end{lemma}
\begin{proof}
For $l\in \mathbb{S}^n$ and $r>0$ let
\begin{equation}
\theta^U(rl):= \psi^U(e^r l),\quad  \theta^L(rl):=\psi^L(e^rl).
\label{thet}
\end{equation}
Note that  the functions $L$ and $U$  satisfy
$$
U(x_2)-U(x_1)=\int_{x_1}^{x_2} {2\over x}L(x)\, dx, \quad x_1<x_2.
$$
Then, taking two  parallel  vectors $\xi_1$ and $\xi_2$, and applying the above relation with $x_1=\xi_1\cdot u$, $x_2= \xi_2\cdot u$, where  $u\in \rn$ and $\|\xi_1\|\leq \|\xi_2\|$, we derive by the Fubini theorem
\begin{equation}
\begin{split}
\psi^U(\xi_2)-\psi^U(\xi_1) &=\int_\rn \big[U((\xi_2,u))-U((\xi_1,u))\big] \mu(du)
%\\&
%=\int_\rn \int_{(\xi_1,u)}^{(\xi_2,u)} \frac{2}{r} L(r)dr \mu(du)
\\&
=\int_\rn \int_{\|\xi_1\|}^{\|\xi_2\|} \frac{2}{r} L(r(l\cdot u))dr \mu(du)
\\&
=\int_{\|\xi_1\|}^{\|\xi_2\|} \tfrac{2}{r} \psi^L(lr)\,dr,  \label{psipm2}
\end{split}
\end{equation}
where  $l:=\xi_1/\|\xi_1\|$. Thus, by \eqref{psipm2} and condition \textbf{A} we have
\begin{equation}
\theta^U(\xi_2) -\theta^U(\xi_1)\geq \frac{2}{\beta} \int_{\|\xi_1\|}^{\|\xi_2\|} \theta^U(vl)dv,\label{tu}
\end{equation}
implying   that
$e^{-\frac{2}{\beta} \|\xi_2\|} \theta^U(\xi_2)\geq
e^{-\frac{2}{\beta} \|\xi_1\|} \theta^U(\xi_1)$.
Thus,
$$
 \psi^U(e^{\|\xi_2\|} l ) = \theta^U(\xi_2)\geq c_1e^{\frac{2}{\beta} \|\xi_2\|},
$$
 where $c_1:=  e^{-\frac{2}{\beta} \|\xi_1\|}  \inf_{l\in \mathbb{S}^n} \theta^U(\xi_1)>0$. Taking $ \inf_{l\in \mathbb{S}^n}$  in the left-hand side of the preceding inequality, we arrive at \eqref{eq-gr}.
\end{proof}

The proof of Theorem~\ref{t-main1} and Proposition~\ref{prop1} rely on the following lemma.
\begin{lemma}\label{le2}
For each $t>0$ the variable $\overline{Z}_t$ possesses the density $\bar p_t(x)$, which satisfies
\begin{equation}
\left| \frac{\partial^N}{\partial x_1^{k_1}\ldots \partial
x_n^{k_N}} \overline{p}_t(x)\right| \leq b_1 \rho_t^{N+n} e^{-b_2
\rho_t\|x\|}, \quad x\in \rn, \quad t\in (0,t_0], \label{e20}
\end{equation}
for any $N\geq 0$,  $k_i\geq 0$,
$i=1\ldots n$, such that $k_1+\dots+k_n=N$.
\end{lemma}
\begin{proof}
 For $n=1$ we have
$$
t\mu\{ u:\, \rho_t \|u\|\geq 1\} \leq t\psi^*(\rho_t)=1.
$$
 For $n\geq 2$ the situation is similar, but  a bit more complicated: since
\begin{equation}\label{qr}
\begin{split}
\mu\{ u:\, \|u\|\geq r\}&\leq \sum_{i=1}^n \mu\{ u:\, |u_i|\geq r\} +\mu\{ u:\, \|u\|\geq r, \, |u_i|<r, \, i=1,\ldots,n\}\\
&\leq \sum_{i=1}^n \mu\{ u:\, |u_i|\geq r\} + \mu\{ u:\, r/2\leq  |u_i|<r, \, i=1,\ldots,n\}\\
&= \sum_{i=1}^n \mu\{ u:\, |u_i|\geq r\} \\
&\quad + \mu\{ u:  \, |u_i|\geq  r, \, 1\leq  i\leq n\}- \mu\{u: \, |u_i|\geq  r/2,\,  1\leq  i\leq n\}\\
&\leq \sum_{i=1}^n \mu\{ u:\, |u_i|\geq r\} + \mu\{ u: \exists i: \, |u_i|\leq r\}\\
&\leq (n+1) \psi^{*} (1/r),
\end{split}
\end{equation}
we arrive at $t\mu\{ u:\, \rho_t \|u\|\geq 1\}\leq n+1$.  Therefore,
\begin{equation}
\begin{split}
 \RR \psi_t(\xi) &= t\RR\psi (\xi)- t \int_{\rho_t \| u\|\geq 1} (1-\cos (\xi \cdot u)) \mu(du)\geq t\RR\psi (\xi) -2t \mu\{ u:\, \rho_t \|u\|\geq 1\}
\\&
= t\RR\psi (\xi) -2(n+1) \geq
t\left(\frac{1-\cos 1}{\beta}\right) \psi^U(\xi)  - 2(n+1)
\geq c_1t \|\xi\|^{2/\beta} - 2(n+1). \label{e1}
\end{split}
\end{equation}
where in the last line we used \eqref{eq-gr}.
  Thus, by Lemma~\ref{growth}  the variable
$\overline{Z}_t$ possesses a distribution density
$\overline{p}_t\in C_b^\infty(\rn)$, and  for any $N\geq 0$,
$k_1+\ldots +k_n=N$, we have
\begin{equation}
 \frac{\partial^N}{\partial x_1^{k_1}\dots \partial x_n^{k_n}} \overline{p}_t(x)= (2\pi)^{-n}
 \int_\rn (-i x_1)^{k_1} \dots (-i x_n)^{k_n} e^{-ix\cdot \xi - \psi_t(\xi)}d\xi. \label{e21}
\end{equation}
Put $H(t,x,z):=-iz\cdot x -\psi_t(z)$. Note that by the structure of $\psi_t$ the function $H(t,x,z)$
 can be extended analytically (with respect to $z$)  to $\Cc^n$. Applying the Cauchy theorem, we derive
 \begin{align*}
\frac{\partial^N}{\partial x_1^{k_1}\ldots \partial x_n^{k_n}}
\overline{p}_t(x)&= (2\pi)^{-n}
 \int_\rn (-i z_1)^{k_1} \dots (-i z_n)^{k_n} e^{H(t,x,z)}dz
 \\&
 =
 (2\pi)^{-n}
 \int_\rn \prod_{j=1}^N(-i y_j+ \eta_j )^{k_j} e^{x\cdot\eta  - i x\cdot y- \psi_t(y+i\eta)}dy.
\end{align*}
 for any $\eta\in \rn$ satisfying  $\|\eta\|\leq \rho_t$. Since the  proof of the above equality repeats line by line  the proof   of \cite[Lemma~3.4]{KK12a},
see also \cite{KS10a} and \cite{KK11} for the $n$-dimensional case, we omit the details.

For $\|\eta\|\leq \rho_t$ we have
\begin{align*}
\RR H(t,x,y+i\eta)&= x\cdot \eta -t \int_{\rho_t \| u\|\leq 1} \left( 1-\eta \cdot u- e^{-u \cdot \eta}\right) \mu(du)
\\&
\quad - t \int_{\rho_t \| u\|\leq 1}  e^{-\eta \cdot u} (1-\cos (y\cdot u)) \mu(du)
\\&
\leq x\cdot \eta - \psi_t(i\eta)-e^{-1}\RR \psi_t(y),
\end{align*}
which implies the upper bound
\begin{equation}
\left|\frac{\partial^N}{\partial x_1^{k_1}\ldots \partial x_n^{k_n}}
\overline{p}_t(x) \right| \leq c_2 e^{\eta \cdot x -\psi_t(i\eta)}
\int_\rn (\|\eta\|+\|y\|)^N e^{-e^{-1} \RR\psi_t(y)}dy. \label{e22}
\end{equation}
Put
$$
c:= \sup_{|s|\leq 1} \Big|\frac{1-s-e^{-s}}{s^2}\Big|, \quad s\in \real.
$$
Using again the inequality $\|\eta\|\leq \rho_t$ and that $\{ u:\, \rho_t \|u\|\leq 1\}\subset \{u: \, |\eta\cdot u|\leq 1\}$,  we derive
$$
-\psi_t (i\eta) \leq ct \int_{\rho_t \|u\|\leq 1} |\eta\cdot u|^2 \mu(du)\leq ct \psi^*(\rho_t)
= c.
$$
Thus,  taking in \eqref{e22} the vector $\eta$ with coordinates
$\eta_i =-\rho_t sign\,  x_i$, $i=1\ldots n$, we get
\begin{equation}
\left|\frac{\partial^N}{\partial x_1^{k_1}\ldots \partial x_n^{k_n}}
\overline{p}_t(x) \right|\leq c_3 e^{-\rho_t \|x\|} \int_\rn
(\rho_t^N + \|y\|^N) e^{-e^{-1}\RR \psi_t(y)}dy,
\end{equation}
where $c_3 \equiv c_3(n,N)>0$  is some constant.  Recall that in  \eqref{e1} we proved that
$\RR \psi_t(y)\geq  t c_4  \psi^U(y)  - 2$, where $c_4:= \frac{1-\cos 1}{\beta}$.
Therefore,
$$
\left|\frac{\partial^N}{\partial x_1^{k_1}\ldots  \partial
x_n^{k_n}} \overline{p}_t(x) \right| \leq c_5 e^{-\rho_t \|x\|}
\sup_{l\in \Sss^n} \left( \rho_t^N I_{n-1} (t,c_6,l)+
I_{N+n-1}(t,c_6, l)\right),
$$
where $c_6 := e^{-1} c_4$, and
\begin{equation}
I_k (t,\lambda,l):= \int_0^\infty e^{- \lambda t \theta^U(vl) +(k+1)v}dv, \quad k\geq 0.  \label{itk}
\end{equation}
To finish the proof we need to show that
\begin{equation}
\sup_{l\in \Sss^n} I_k(t,\lambda,l)\leq c_7\rho_t^{k+1}. \label{itk2}
\end{equation}
We get
\begin{align*}
\sup_{l\in \Sss^n} I_k(t,\lambda,l)&= \rho_t^{k+1}  \underset{l\in \Sss^n}{\sup}\int_0^\infty e^{-\lambda t [\theta^U (vl)-\theta^U(v_t l)]+(k+1)(v-v_t)- \lambda t \theta^U(v_t l)}dv
\\&
\leq \rho_t^{k+1} \int_0^\infty
e^{-\lambda t \underset{l\in \Sss^n}{\inf}[\theta^U (vl)-\theta^U(v_t l)]+(k+1)(v-v_t)-\lambda t\underset{l\in \Sss^n}{\inf} \theta^U(v_t l) }dv
\\&
\leq \rho_t^{k+1} \left[\int_0^{v_t}+ \int_{v_t}^\infty\right]
e^{-\lambda t \underset{l\in \Sss^n}{\inf}[\theta^U (vl)-\theta^U(v_t l)]+(k+1)(v-v_t)}dv,
\end{align*}
where $v_t:= \ln \rho_t$, and in the last line we used that $\theta^U$ is non-negative.
To estimate the first integral
observe that
\begin{equation}
\int_0^{v_t}
e^{-\lambda t [\theta^U (vl)-\theta^U(v_t l)]+(k+1)(v-v_t)}dv
\leq e^{\lambda t \psi^U(l\rho_t)}  \int_0^{v_t} e^{(k+1)(v-v_t)}dv\leq
\frac{e^\lambda}{k+1}.\label{e10}
\end{equation}
Using condition \textbf{A} and \eqref{tu}  we derive
\begin{align*}
[\theta^U(vl)-\theta^U(v_t l)]&= 2 \int_{v_t}^v \theta^L(rl)dr\geq \frac{2}{\beta} \int_{v_t}^v \theta^U(rl)dr
\\&
=\frac{2}{\beta} \theta^U(v_t l) (v-v_t)+\frac{4}{\beta}\int_{v_t}^v \int_{v_t}^r \theta^L(sl)dsdr
\\&
\geq
\frac{2}{\beta} \theta^U(v_t l) (v-v_t)+\frac{4}{\beta^2}\int_{v_t}^v \int_{v_t}^r \theta^U(sl)dsdr
\\&
\geq \frac{2}{\beta} \theta^U(v_t l) (v-v_t)+\frac{4}{\beta^2}\theta^U(v_t l) (v-v_t)^2.
\end{align*}
Further, by  \eqref{eq1} and condition \textbf{A}  we have
\begin{equation}
t \inf_{l\in \Sss^n} \theta^U(v_t l)\geq \tfrac{t(1-\cos 1)}{2\beta} \sup_{l\in \Sss^n} \psi^U(\rho_t l) =  \tfrac{t(1-\cos 1)}{2\beta} \sup_{l\in \Sss^n} \psi^* (\rho_t) = \tfrac{1-\cos 1}{2\beta} , \label{e11}
\end{equation}
 implying
$$
t\inf_{l\in \Sss^n} [\theta^U(vl)-\theta^U(v_t l)] \geq
b(v-v_t)+2b\beta^{-1} (v-v_t)^2,
$$
where $b=\tfrac{1-\cos 1}{\beta^2} $. Thus,
\begin{equation}
\int_{v_t}^\infty e^{-t \lambda \inf_{l\in \Sss^n}[\theta^U (vl)-\theta^U(v_t l)]+(k+1)(v-v_t)}dv\leq \int_0^\infty e^{(k+1)w-b\lambda  w - \frac{2b\lambda}{\beta} w^2}dw<\infty.
\label{e12}
\end{equation}
Combining \eqref{e10} and \eqref{e12} we get \eqref{itk2}, which finishes the proof.
\end{proof}

If the L\'evy measure $\mu$  is symmetric, one can refine the upper estimate in \eqref{e20}.
\begin{lemma}\label{le3}
Let condition \textbf{A} hold true, and suppose in addition that the
L\'evy measure $\mu$ is symmetric. Then for any $N\geq 0$, and any
$k_i\geq 0$, $i=1\ldots n$, $k_1+\ldots +k_n=N$, we have
\begin{equation}
\left|\frac{\partial^N}{\partial x_1^{k_1}\ldots \partial x_n^{k_N}}
\overline{p}_t(x)\right|\leq b_1 \rho_t^{N+n} e^{-b_2 \rho_t \|x\|
\ln (\rho_t\|x\|+1)}, \quad x\in \rn,\quad t\in (0,t_0]. \label{e30}
\end{equation}
\end{lemma}
\begin{proof}
By the same argument as in \cite[Lemma~3.6]{KK12a} we have for \emph{any}  $\eta\in \rn$
\begin{equation}
 \left|\frac{\partial^N}{\partial x_1^{k_1}\ldots \partial x_n^{k_N}} \overline{p}_t(x)\right|\leq (2\pi)^{-n}  e^{\eta\cdot x- \psi_t (i\eta)} \int_\rn (\|y\|+\|\eta\|)^N e^{-\RR \psi_t(y)}dy. \label{sym1}
 \end{equation}
 By Lemma~\ref{le2}, the integral in \eqref{sym1} is estimated from above by $c_1 (\|\eta\|^N \rho_t^n + \rho_t^{N+n})$, where $c_1>0$ is some constant. For $\psi_t(i\eta)$ we have
 \begin{align*}
 -\psi_t (i\eta)&= t \int_{\rho_t \|u\|\leq1} [\cosh (\eta \cdot u)-1]\mu(du)= t\theta\big(\|\eta\|/\rho_t\big) \int_{\rho_t \|u\|\leq1}
 (\eta \cdot u)^2 \mu(du)
 \\&
 \leq t\theta\big(\|\eta\|/\rho_t\big)\big(\|\eta\|/\rho_t\big)^2 \sup_{l\in \Sss^n} \int_{\rho_t \|u\|\leq 1} \rho_t^2 (l \cdot u)^2 \mu(du)
 \\&\leq  \Big(\cosh(\|\eta\|/\rho_t)-1 \Big) t \psi^*(\rho_t)
 \\&
 = \cosh(\|\eta\|/\rho_t)-1,
 \end{align*}
 where $\theta(s):= s^{-2} \big(\cosh s-1\big)$, $s\geq 0$, is increasing. Since sofar $\eta$ was arbitrary, take $\eta$ with coordinates satisfying  $sign \,\eta_i = - sign\, x_i$,
 $i=1\ldots n$. Then
\begin{equation}
\left|\frac{\partial^N}{\partial x_1^{k_1}\ldots \partial x_n^{k_N}}
\overline{p}_t(x)\right|\leq c_2 \rho_t^{N+n}e^{
-\|x\|\|\eta\|+\cosh(\|\eta\|/\rho_t)}. \label{e32}
\end{equation}
Minimizing the expression under the exponent in \eqref{e32} in $\|\eta\|$, we arrive at \eqref{e30}.
\end{proof}

\begin{proof}[Proof of Theorem~\ref{t-main1}]
\emph{Upper bound.} The proof of the upper bound follows from Lemmas~\ref{growth},  \ref{le2}, and representation \eqref{conv1}.

\emph{Lower bound.}  From Lemma~\eqref{le2} we know that the function $\overline{p}_t(x)$ is  continuous in $x$,  and bounded from above by $b_1\rho_t^n$.  Without loss of generality we may assume that
$\int_{\rho_t \|x\|\leq 1} \overline{p}_t(x)dx\geq 1/2$.
Then
$$
1/2 \leq  \int_{\rho_t \|x\|\leq L} \overline{p}_t(x)dx\leq \frac{w_n L^n}{\rho_t^n} \max_{x\in \rn} \overline{p}_t(x),
$$
where $w_n$ is the volume of a unit ball in $\rn$.  Let $x_t$ be the "smallest" in the lexicographical sense point in which the maximum of  $\overline{p}_t(x)$ is achieved.   For the off-diagonal lower bound we get using the Taylor formula:
\begin{equation}
\begin{split}
\overline{p}_t(x) &\geq \overline{p}_t(x_t)-\left| \sum_{i=1}^n (x-x_t)_i \int_0^1 \frac{\partial}{\partial x_i}\overline{p}_t(x_t+r(x-x_t)) dr  \right|
\\&
\geq \overline{p}_t(x_t) - \left( \sum_{i=1}^n \int_0^1 \big| \frac{\partial}{\partial x_i} \overline{p}_t(x_t+r(x-x_t)) \big|^2 dr\right)^{1/2} \|x-x_t\|
\\&
\geq  \frac{1}{2 w_n L^n}\rho_t^n - c_1(n)   \rho_t^{n+1} \|x-x_t\|
\\&
= c_2(n) \rho_t^n \big( 1-c_3(n)\rho_t \|x-x_t\|\big),
\end{split}
\end{equation}
where in the second line form below  we used the on-diagonal estimate
$$
\left|\frac{\partial}{\partial y_i} \overline{p}_t(y)\right| \leq  c(n)\rho_t^{n+1}.
$$
\end{proof}

\section{Bell-like estimates}\label{subb}

%\subsection{Settings and Results}
In this section we discuss some particular cases  in which  we pose  more restrictive assumptions on the regularity of the tail of the L\'evy measure.  We show that under certain assumptions  it is possible
to write  more explicit upper and lower estimates for $p_t(x)$.  At the same time, we emphasize that although such  estimates  can be more explicit, they suppress the vital information about the transition probability  density, given by the compound kernel estimates. Moreover, as we will see below,  a bell-like estimate may heavily depend on the space dimension.

We begin with some notions on \emph{sub-exponential distributions} in the multi-dimensional setting, see \cite{Om06} and  \cite{OMS06} for more details.  We keep the notation of Theorem~\ref{t-main1}.
\begin{definition}\cite{Om06}
We say that $G$ is a \emph{sub-exponential} distribution on $\rn$ (and write $G \in \LL(\rn)$) if for all $x\in \rn$ such that  $\min_i x_i<\infty$, we have
\begin{equation}
\lim_{t\to\infty} \frac{1-G^{*2}(tx)}{1-G(tx)}=1. \label{sub1}
\end{equation}
\end{definition}
Theorem below generalizes the one-dimensional result, proved in \cite{KK12a}.
\begin{theorem}\label{Omey}
Let  condition \textbf{A} hold true,
and suppose that there exist a   distribution function  $G\in \LL(\rn)$, such that
\begin{equation}
t\mu\Big(\{u: \|\rho_t u\|>\|v\|\}\Big)\leq C(1-G(v)), \quad \|v\|\geq 1, \quad t\in (0,t_0], \label{dens02}
\end{equation}
where $C>0$ is some constant, independent of $t$. Then  for every $t_0>0$ there exist some constant $C_1>0$, such that
\begin{equation}\label{bell1}
p_t(x+a_t)\leq C_1  \rho_t^n \left( f_{upper}(\rho_t x)+ 1-G(x\rho_t)\right), \quad x\in \rn, \quad t\in (0, t_0],
\end{equation}
where $f_{upper}$ is defined by \eqref{fup}. % (respectively, by \eqref{fup-sym}, if $\mu$ is symmetric).
 If the inequality \eqref{dens02}  holds true with the sign $\geq$, then
\begin{equation}
p_t(x+a_t -x_t) \geq C_2 \rho_t^n \big( f_{lower}(\rho_t x)+1-G(\rho_t x)\big), \quad x\in \rn, \quad t\in(0,t_0],
\end{equation}
where $C_2>0$ is some constant, and $f_{lower}$ is defined in \eqref{flower}.
\end{theorem}

In \cite{KK12a} we  proved a version of Theorem~\ref{Omey} in the case when the measure $\mu$ is absolutely continuous, and the density is sub-exponential in the sense of \cite{Kl89}. Up to our knowledge sub-exponential \emph{densities} are not studied in the multi-dimensional case, see, however, \cite{Om06} for a brief comment.  We strongly believe that the result analogous to those proved in \cite{KK12a} also can be proved in the  multi-dimensional setting, after establishing the necessary properties of sub-exponential densities analogous to those presented in \cite{Kl89}. However, it is possible to prove a version of Theorem~\ref{Omey} under the assumption of a power decay of the L\'evy density.
\begin{theorem}\label{den}
Let  condition \textbf{A} hold true.
Suppose that $\mu(du)=m(u)du$, and  for $\|u\|\geq 1$ we have  the estimate
\begin{equation}
t\rho_t^{-n} m(u\rho_t^{-1}) \leq \|u\|^{-n-b},
 \quad t\in (0,t_0],\label{dens01}
\end{equation}
where  $b>0$.  Then
\begin{equation}\label{bell11}
p_t(x+a_t)\leq c_1\frac{\rho_t^n}{(1+\rho_t \|x\|)^{n+b}},  \quad x\in \rn, \quad t\in (0, t_0].
\end{equation}
If the inequality \eqref{dens01}  holds true with the sign $\geq$, then
\begin{equation}
p_t(x+a_t -x_t) \geq c_2\frac{\rho_t^n}{(1+\rho_t\|x\|)^{n+b}},  \quad x\in \rn, \quad t\in (0, t_0].
\end{equation}
\end{theorem}

%\subsection{Proofs of Theorems \ref{Omey} and \ref{den}}

 The proof of Theorem~\ref{Omey} relies on the results obtained in \cite{Om06}.
In order to make the presentation self-contained, we quote these results below.

  It is shown in \cite[Theorem~7, Corollary~11]{Om06}  that for a distribution function $G$ the conditions

\medskip

\textbf{G1}. For $\forall a\,, x\in \rn$, $a\geq 0$, $x\geq 0$, such that $\min_i x_i<\infty$,
$\underset{t\to\infty}{\lim} \frac{1-G(tx-a)}{1-G(tx)}=1$;

\textbf{G2}. All marginals $G_i$ of $G$ are sub-exponential  (i.e., $G_i\in \LL(\real)$),

\medskip

\noindent are equivalent to $G\in \LL(\rn)$, and imply
%\footnote{ In fact, in \cite{Om06} this result was proved under more general $H$, but to keep the exposition
%reasonably tight we use the simplified version, which sufficient in our situation.}
that for $x\geq 0$, $\min x_i<\infty$, and $a\in \rn$, $a\geq 0$, one has
\begin{equation}
\lim_{t\to \infty} \frac{1-H(tx-a)}{1-G(tx)}=\lambda, \label{H1}
\end{equation}
where
\begin{equation}
H(x)= \sum_{k=1}^\infty \frac{\lambda^k}{k!} G^{*k}(x), \quad \lambda\in (0,\infty). \label{H2}
\end{equation}
We also need \cite[Theorem~10]{Om06}, which  states that if the distribution function $G$ satisfies \textbf{G1} and \textbf{G2}, and the distribution functions $R$ and $F$ are such that
\begin{equation}
\lim_{t\to\infty} \frac{1-F(tx-a)}{1-G(tx)}=\alpha, \label{F1}
\end{equation}
\begin{equation}
\lim_{t\to\infty} \frac{1-R(tx-a)}{1-G(tx)}=\beta,\label{R1}
\end{equation}
for some $\alpha,\,\beta\in \real$, and any $a,\, x\in \rn$, $a,\, x\geq 0$,  $\min_i x_i<\infty$, then
\begin{equation}
\lim_{t\to \infty} \frac{1-R*F(tx-a)}{1-G(tx)}=\alpha+\beta. \label{H3}
\end{equation}
\begin{proof}[Proof of Theorem~\ref{Omey}]
 By \eqref{H2} we have
\begin{equation}
p_t(x)\leq \rho_t^n f_{upper}(x\rho_t) + c_1 \rho_t^n \int_{\|v\|\geq 1}
f_{upper}(x\rho_t -v) G(dv). \label{dens20}
\end{equation}
 Note that  for any $c>0$ the tail of a sub-exponential distribution in $\real$   decays  slower than $e^{-c |y|}$ as $|y|\to \infty$, (see \cite{Kl89}, also the comment in \cite{KK12a}), which   implies  that for any $c>0$  the tail of a  sub-exponential distribution in $\rn$ decays slower than $e^{-c\|x\|}$ as $\|x\|\to\infty$. Hence, for $R(x)=1-f_{upper}(x)$ we have \eqref{R1} with $\beta=0$. Thus, by sub-exponentiality of $G$ we have the relation  \eqref{H3} with $\alpha=1$, $\beta=0$, i.e.
\begin{align*}
\lim_{s\to \infty } \frac{\int_{\|v\|\geq 1} f(xs -v)dG(v)}{1-G(sx)}=1.
\end{align*}
Since  $\rho_t\to \infty$ as $t\to 0$, we finally derive  \eqref{bell1} for $t$ small enough.

Similar argument works for the lower bound: in this case we take $R(x)= 1-f_{lower}(x)$.
\end{proof}
\begin{proof}[Proof of Theorem~\ref{den}]
Let $q(v):= (1+\|v\|)^{-n-b} $, and put $Q(v):= \sum_{k=1}^\infty \tfrac{q^{*k}(v)}{k!}$, $v\in \rn$. By Theorem~\ref{t-main1} and \eqref{dens01} we get
\begin{equation}
p_t(x)\leq c \rho_t^n \Big(f_{upper}(x\rho_t)+ \int_\rn f_{upper}(x\rho_t -v)Q(v)dv\Big). \label{den1}
\end{equation}
Let us estimate $Q(v)$.  We have:
\begin{align*}
q^{*2} (w)& = \int_\rn \frac{1}{(1+\|v\|)^{n+b}(1+\|w-v\|)^{n+b}} dv
\\&
= \Big[\int_{\{\|w-v\|\leq 2^{-1} \|w\|\}} +\int_{\{\|w-v\|\geq 2^{-1} \|w\|\}} \Big]
 \frac{1}{(1+\|v\|)^{n+b}(1+\|w-v\|)^{n+b}} dv
 \\&
 =I_1+I_2.
\end{align*}
To estimate  $I_1$ observe that   if $\|w-v\|\leq 2^{-1} \|w\|$, then $\|w\|\leq \|v\|\leq 3/2 \|w\|$, or $1/2 \|w\|\leq \|v\|\leq \|w\|$,  implying $\tfrac{1}{1+\|v\|}\leq \tfrac{2}{2+\|w\|}$. Therefore,
$$
I_1\leq \Big( \frac{2}{2+\|w\|}\Big)^{n+b} \int_\real \frac{1}{(1+\|v\|)^{n+b}}dv\leq
c  \Big( \frac{2}{1+\|w\|}\Big)^{n+b}.
$$
Analogously, if $\|w-v\|\geq 2^{-1} \|w\|$, then $\tfrac{2}{2+\|w\|}\geq \tfrac{1}{1+\|w-v\|}$, implying
$$
I_2\leq \Big( \frac{2}{2+\|w\|}\Big)^{n+b} \int_\real \frac{1}{(1+\|v\|)^{n+b}}dv\leq
c  \Big( \frac{2}{1+\|w\|}\Big)^{n+b}.
$$
Thus, there exists a constant $C>0$ such that $q^{*2}(v) \leq C q(v)$. By induction, $q^{*k}(v)\leq C^{k-1} q(v)$, implying $Q(v)\leq c_1  q(v)$, $v\in \real$.
 Finally, observe that
\begin{align*}
\int_\real f_{upper}( x-v)Q(v)dv&= \Big[\int_{\|x-v\|\geq 2^{-1}\|x\|} + \int_{\|x-v\|\leq 2^{-1}\|x\|}\Big] f_{upper}(x-v)Q(v)dv
\\&
\leq c_2 f_{upper}(x/2) +c_3 Q(x)
\leq c_4 Q(x).
\end{align*}
Thus, we arrive at
$$
p_t(x)\leq c_5 \frac{\rho_t^n}{(1+\rho_t\|x\|)^{n+b}},
$$
which proves the first part of the theorem. The same argument applies for the lower bound.
\end{proof}

\section{Examples}\label{exam}

\begin{example}\label{exa1}

 Let $Z_t$ be an $\alpha$-stable process, $\alpha\in (0,2)$, with the L\'evy measure $\mu(du)=c_\alpha\|u\|^{-n-\alpha} du$, and the drift vector $b\in \rn$.   One can easily verify that condition \textbf{A} is satisfied, and
 $\rho_t=t^{-1/\alpha}$.  Applying  Theorem~\ref{den},  we arrive at
 $$
 p_t(x+bt)\asymp  t^{-n/\alpha}\wedge \frac{t}{\|x\|^{1+\alpha}}\asymp t^{-n/\alpha} f(t^{-1/\alpha}\|x\|), \quad x\in \rn, \quad t\in (0,t_0],
$$
where
\begin{equation}\label{eq-exa11}
f(z)=1\wedge z^{-\alpha-n}, \quad z>0,
\end{equation}
and  for the lower bound we used that due to the symmetry of the L\'evy measure we have $x_t=0$. Note that by the structure of $\mu$  the above estimates hold true for all $t>0$, $x\in \rn$, and coincides in the case $b=0$ with the well-known estimate for the transition probability density of a symmetric  $\alpha$-stable process.

 Observe that for $1<\alpha<2$ we have
 $$
 t^{-1/\alpha} \|x-tb\|\geq t^{-1/\alpha} - t^{1-1/\alpha}\|b\|\geq t^{-1/\alpha}\|x\|-c\|b\|,\quad t\in (0,t_0].
$$
Thus, for such $\alpha$ we arrived at
$$
p_t(x)\asymp   t^{-n/\alpha} f(t^{-1/\alpha} \|x\|), \quad t\in (0,t_0], \quad x\in \rn.
$$
\end{example}

\begin{example}\label{exa2} Consider a "discretized version" of an $\alpha$-stable L\'evy measure in $\rn$.
Let $m_{k,\upsilon}(dy)$ be a  uniform distribution   on a sphere $\mathbb{S}_{k,\upsilon}$ centered at $0$ with radius $2^{-k\upsilon}$, $\upsilon>0$, $k\in \mathbb{Z}$.
Consider  a L\'evy process with characteristic exponent of the form \eqref{psi}, where
 $$
\mu(dy)=\sum_{k=-\infty}^\infty 2^{k\gamma}m_{k,v}(dy), \quad 0<\gamma<2\upsilon,
$$
 and  some drift coefficient $a\in \rn$.  Let us check that in this case $\psi^U(\xi) \asymp \psi^L(\xi)\asymp \|\xi\|^\alpha$, where  $\alpha=\gamma/\upsilon$.

Let $k_0:= \upsilon^{-1} \log_2\|\xi\|$. We have
 \begin{align*}
 \psi^U(\xi)&\leq \int_\rn
 (\|\xi\|^2\|y\|^2\wedge 1) \mu(dy)
 \\&
 =  \|\xi\|^2\int_{\|y\|\leq /\|\xi\|}  \|y \|^2 \mu(dy)+ \int_{\|y\|> 1/\|\xi\|}\mu(dy)
 \\&
 = \|\xi\|^2 \sum_{k\geq k_0}  2^{\gamma k-2k\upsilon}+ c_1 \sum_{k\leq k_0} 2^{\gamma k}
 \\&
 \leq  \|\xi\|^2  2^{k_0 (\gamma -2v)} \sum_{k\geq k_0} 2^{-(k-k_0)(2\upsilon-\gamma)} +c_1 + 2^{\gamma k_0} \frac{1-2^{-\gamma k_0}}{1-2^{-\gamma}}
 \\&
 \leq \frac{2^{2\upsilon-\gamma}}{2^{2\upsilon-\gamma}-1}  \|\xi\|^2 2^{\frac{2\upsilon-\gamma}{\upsilon}\log_2 \|\xi\|}+ c_2 2^{\frac{\gamma}{\upsilon} \log_2 \|\xi\|}
 \leq c_3\|\xi\|^\alpha.
 \end{align*}
The above calculations and the inequality  $(1-\cos 1)\psi^L(\xi)\leq \int_\rn (1-\cos (\xi \cdot y))\mu(dy)$  imply that
\begin{align*}
\psi^L(\xi)&\leq c_4 \psi^U(\xi) \leq c_5 \|\xi\|^\alpha.
\end{align*}
For the lower bound we have
\begin{align*}
\psi^L(\xi) &\geq \int_{\|y\|\geq 1/\|\xi\|} |\xi\cdot y|^2 \mu(dy)\geq m_{k_0,v}\{ l\in \mathbb{S}_{k_0,\upsilon}: \, |\cos(l_\xi \cdot l)|>\epsilon\}
\|\xi\|^2 2^{k_0 (\gamma-2\upsilon)}
\\&
= c_6 \|\xi\|^\alpha,
\end{align*}
where $l_{\xi}:= \xi/\|\xi\|$, implying
$$
\inf_{\|l\|=1} \psi^L(\|\xi\| l ) \geq  c\|\xi\|^\alpha.
$$
Thus, condition \textbf{A} is satisfied, and $\psi^L(\xi)\asymp \psi^U(\xi)\asymp \|\xi\|^\alpha$, which in turn gives
 $\rho_t \asymp t^{-1/\alpha}$.

 Note that for $\|x\|>1$ we have
\begin{align*}
 t\mu\Big(\{u: \rho_t\| u\|>\|x\|\}\Big) =t \sum_{n \leq n(t,x)} 2^{\gamma n} \leq  Ct 2^{\tfrac{\gamma}{\upsilon} \log_2 (\rho_t/\|x\|)}
 = C \|x\|^{-\gamma/\upsilon}= C\|x\|^{-\alpha},
\end{align*}
 where $n(t,x):= \tfrac{1}{\upsilon}\log_2 (\rho_t/\|x\|)$.  Therefore, condition \eqref{dens02} of Theorem~\ref{Omey} holds true with $1-G(x)=\|x\|^{-\alpha}$, $\|x\|\geq 1$. By this theorem we have the following estimate for the respective transition probability density:
 \begin{equation}
p_t(x+at) \leq c_1 t^{-n/\alpha} f(t^{-1/\alpha}\|x\|)\label{eq-exa2}
\end{equation}
where
\begin{equation}
f(z)=1\wedge z^{-\alpha}, \quad z>0.\label{eq-exa21}
\end{equation}
However, as one may notice,  such upper estimate is informative only in the case $n=1$ and $1<\alpha<2$, see  \cite{KK12a} for the detailed analysis. In the  other cases the upper bound is not integrable!   On the other hand,  Theorem~\ref{t-main1}  together with Proposition~\ref{aux-theo-sym} provides that the transition probability density satisfies the upper compound kernel estimates with parameters $( t^{-1/\alpha}, f_{upper},  t^{-1/\alpha}, \Lambda_t)$, with
$$
f_{upper}(x)=b_1e^{-b_2 \|x\|\log(1+\|x\\)}, \quad \text{and} \quad \Lambda_t(du)= t \I_{\{\|u\|\geq  t^{1/\alpha}\}}\mu(du).
   $$
In this case the obtained upper bound is integrable.
\end{example}

\begin{remark} The above example  illustrates that  even  if the (re-scaled)  L\'evy  measure can be dominated by a reasonably good function, the explicit upper  estimate  obtained in Theorem~\ref{Omey} can be extremely inexact.  Heuristically, the condition \eqref{dens02} is imposed on the tail of the  re-scaled \emph{measure}, which suppresses its  intrinsic  behaviour.  See, however, \cite{KS13b} for another approach in a similar situation.   On the other hand, the condition on the behaviour of the \emph{density} can lead to  adequate results, as we saw in Example~\ref{exa1}. Possibly, one can modify the  assumption Theorem~\ref{Omey} and get more reasonable explicit estimates, but in fact it is not needed, since the compound kernel estimates obtained in Theorem~\ref{t-main1} already contain the information, sufficient for  many applications, see  \cite{KK13a} and \cite{KK13b}.
\end{remark}

\textbf{Acknowledgement.} %The author would like to thank the referee for helpful suggestions and remarks.
The author is very grateful to Alexei Kulik for fruitful discussions.  The DFG Grant Schi~419/8-1, and   the Scholarship of the President of Ukraine for young scientists (2011-2013) are gratefully acknowledged.

\end{document}